\def\supp{1} 

\documentclass{article}
\pdfoutput=1
\usepackage{amstext,amssymb,amsmath}
\usepackage{amsthm}
\usepackage{mathtools}
\usepackage{verbatim}
\usepackage{dsfont}
\usepackage{hyperref}
\usepackage[noend]{algorithmic}
\usepackage{algorithm}
\usepackage[round]{natbib}

\newtheorem{theorem}{Theorem}[section]
\newtheorem{proposition}{Proposition}[section]

\newtheorem{lemma}[theorem]{Lemma}
\newtheorem{definition}[theorem]{Definition}
\theoremstyle{remark}
\newtheorem{remark}{Remark}

\begin{document}
\title{A short note on the operator norm upper bound for sub-Gaussian tailed random matrices}

\author{Eric Benhamou \thanks{A.I. Square Connect} \thanks{Lamsade, Paris Dauphine} \thanks{Email:
    \texttt{eric.benhamou@aisquareconnect.com, eric.benhamou@dauphine.eu}}  \and  
Jamal Atif \footnotemark[2] \thanks{Email: \texttt{jamal.atif@dauphine.fr}}  \and 
Rida Laraki \footnotemark[2] \thanks{Email: \texttt{rida.laraki@dauphine.fr}}}
\maketitle

\begin{abstract}
This paper investigates an upper bound of the operator norm for sub-Gaussian tailed random matrices. A lot of attention has been put on uniformly bounded sub-Gaussian tailed random matrices with independent coefficients. However, little has been done for sub-Gaussian tailed random matrices whose matrix coefficients variance are not equal or for matrix for which coefficients are not independent. This is precisely the subject of this paper. After proving that random matrices with uniform sub-Gaussian tailed independent coefficients satisfy the Tracy Widom bound, that is, their matrix operator norm remains bounded by $O(\sqrt n )$ with overwhelming probability, we prove that  a less stringent condition is that the matrix rows are independent and uniformly sub-Gaussian. This does not impose in particular that all matrix coefficients are independent, but only their rows, which is a weaker condition. 
\end{abstract}

\section{Introduction}
Random matrices and their spectra have been under intensive study in many fields. This is the case in Statistics since
the work of \cite{Wishart_1928} on sample covariance matrices, in Numerical Analysis since their introduction by  \cite{VonNeumann_1947} in the 1940s, in Physics as a consequence of the work of \cite{Wigner_1955, Wigner_1958} since the 1950s on in Banach Space Theory and Differential Geometric Analysis with the work of \cite{Grothendieck_1956} in a similar period. 
More recently, in machine learning, the netflix prize (see \cite{wiki:Netflix_prize}) has attracted a lot of attention with a large part of the community investigating recommender systems (see \cite{wiki:recommender_system}) and collaborative filtering methods, which ultimately also rely on random matrices and their eigen and singular values spectra. 

In particular, an interesting and important problem in matrix completion problem has been to investigate where the operator norm is concentrated to be able to make some reasonable assumptions about missing entries. Other important contribution have been the Tracy Widom law, which says that for Wigner matrix, the operator norm is concentrated in the range of $\left[ 2 \sqrt{n} - O(n^{-1/6}), \right.$ $ \left. 2 \sqrt{n} + O(n^{-1/6}) \right]$ 
(see \cite{Tracy_1994}), and the Marchenko–Pastur distribution that describes the asymptotic behavior of singular values of large rectangular random matrices (see \cite{Marcenko_1967}). 

However, most of these results have been derived under the assumptions of independent and identically distributed coefficients. It is natural to ask similar questions about general random matrices whose entries distribution may differ. In particular, to make the question more concrete, we are interested in finding an upper bound of the operator norm of a random matrix whose coefficients are sub Gaussian and see the implied consequence for the matrix coefficients. The paper is organized as follows. In section \ref{intro}, we recall various definitions. In section \ref{main}, we first proved that for independent and uniform sub-Gaussian tailed random squared matrices their operator norm satisfies the Tracy Widom bound, that is, the matrix operator norm for the $L_a, L_b$ norm remains bounded by $O(\sqrt n )$. We see that a less stringent sufficient condition is that the matrix rows  $L_{a}$  norms are uniformly sub-Gaussian and independent. This implies in particular that a matrix with coefficients that are not necessarily independent and sub-Gaussian can still validate an upper bound for the its operator norm of $O(\sqrt n )$ with overwhelming probability.

The condition of independence of rows has already been mentioned in \cite{vershynin_2018} with a similar setting and proof and appeared as early as 2017. Additionally, \cite{Benaych_2018} provided a similar proof in the Hermitian case and pointed kindly to the authors the last two references that authors were not aware of at the time of their writing. This article has at least the merit to be self contained and to focus only on sub-Gaussian random matrix making the presentation shorter and self consistent. But for more details, we advise the reader to refer to the last two references that cover a much wider scope and are respectively 300 and 80 pages long.

\section{Some definitions}\label{intro}
Suppose $\| \cdot \|_a$ and  $\| \cdot \|_b$ are norms on $\mathbb{R}^m$ and $\mathbb{R}^n$, respectively. We can of course generalize easily the concept to norms operating on $\mathbb{C}^m$ and $\mathbb{C}^n$ if we look at matrices with complex number coefficients.

\begin{definition}\label{operatornorm_def}
We define the operator norm of $ \mathbf{X} \in  \mathbb{R}^{m \times n}$, induced by the norms  $\| \dots \|_a$ and  $\| \dots \|_b$, as
\begin{equation}\label{operatornorm}
\|  \mathbf{X}  \|_{a,b} = \operatorname{sup} \{  \|  \mathbf{X}u  \|_{a} \,\, | \,\, \| u \|_{b} \leq 1\} .
\end{equation}
We will denote this norm as $\|  \cdot  \|_{op}$ and we will drop the $a,b$ indices to make things simpler whenever there is no risk of confusion and have the following definition
\begin{equation}\label{operatornorm2}
\|  \mathbf{X}  \|_{op} = \operatorname{sup} \{  \|  \mathbf{X}u  \| \,\, | \,\, \| u \| \leq 1\} .
\end{equation}

\end{definition}

When $\| \cdot \|_a$ and  $\| \cdot \|_b$ are both Euclidean norms, the operator norm of $\mathbf{X}$ is its
\textit{maximum singular value}, and is denoted  $\| \cdot \|_2$:
\begin{equation}
\|  \mathbf{X}  \|_{2} = \sigma_{\text{max}}( \mathbf{X}  ) = ( \lambda_{\text{max}} (\mathbf{X}^T \mathbf{X}  ))^{1/2}.
\end{equation}
where $ \sigma_{\text{max}}( \mathbf{X}  )$ is the maximum singular value of the matrix $\mathbf{X}$ and where $\lambda_{\text{max}} (\mathbf{X}^T \mathbf{X}  )$ is the maximum eigen value of the matrix $\mathbf{X}^T \mathbf{X}$ also defined as $\text{sup} \{ u^T \mathbf{X}^T \mathbf{X}  u\,\|  \,\ \| u \|_2 = 1\}$. 

In the rest of the paper, we will assume to simplify notation that $a = b = 2$ to keep things simple but all results remain the same for any $a, b \geq 1$.

\begin{remark}
For the trivial matrix consisting entirely of single ones, it  has an operator norm of exactly $n$. 
This can be seen easily by taking the vector $u = ( 1/ \sqrt n, \ldots, 1/ \sqrt n)^T$  that gives $ \| \mathbf{X}u   \|_{	2} = n$, and proves that the operator norm should be at least equal to $n$. But the Cauchy-Schwarz inequality proves that it cannot be more than $n$. This vector is the right one to choose for the $L_2$ norm. But using the fact that any norm is equivalent in finite dimension (and that the matrix space is of finite dimension $n^2$), this result is not specific to the $L_2$ norm and is true for any norm. 

Furthermore, the same application of the Cauchy Schwartz proves that the operator norm of any matrix whose coefficients are uniformly bounded by a constant $K$ has an operator norm bounded by $Kn$. In other words, using the Landau notation, any matrix whose entries are all uniformly $O(1)$ has an operator norm of $O(n)$. However, this upper bound does not take into account of any possible cancellations in the matrix $M$. Indeed, intuitively, using the concentration inequality of Hoeffding and Markov, we should expect with overwhelming probability (a notion that we will define shortly) that the operator norm should be bounded by $\sqrt{n}$ rather than $n$ in most cases where matrices coefficients are symmetrically distributed and have tails that are decreasing fast enough, a concept that we will also make more precised shortly with the concept of sub-Gaussian tails.

As for Euclidean norms, the operator norm boils down to computing the maximum singular value and for symmetric matrices, the maximum eigen values, it gives fruitful information about the these two quantities.
\end{remark}

\begin{definition}
A random variable  $\xi$  is called sub-Gaussian if there are non negative constants $B, b > 0$ such that for every $t > 0$,
\begin{equation}\label{sub-Gaussian}
\mathbb{P}(| \xi |>t) \leq B \exp(- b t^2).
\end{equation}
where $\mathbb{P}$ is the probability measure defined on a usual probability space $ \Omega = ( \Omega, \mathcal{B}, \mathbb{P})$. where $ \Omega$ is the ambient sample space, associated with a $\sigma$-algebra $\mathcal{B}$ of subsets of 
$\Omega$.
\end{definition}

\begin{remark}
Sub-Gaussian can be defined in multiple ways. We have used the traditional definition that states that the tails of the variable $\xi$ are dominated by, meaning they decay at least as fast as, the tails of a Gaussian. A more probabilistic way of defining the sub-Gaussian  is to state that a random variable  $\xi$  is called sub-Gaussian with variance proxy $\sigma$ if
\begin{equation}\label{sub-Gaussian2}
\mathbb{P}(| \xi  - \mathbb{E}[X] |>t) \leq 2 \exp(- \frac{ t^2}{2 \sigma^2}).
\end{equation}

Chernoff bound allows to translate a bound on the moment generating function into a tail bound and vice versa. So we should expect to have equivalent definition in terms of moment generating, Laplace transform and many more criteria. Indeed, there are many equivalent definitions ( that can be found for instance in \cite{Buldygin_1980} or \cite{Ledoux_1991})
\begin{itemize}
\item A random variable  $\xi$  is sub-Gaussian.

\item A random variable  $\xi$  satisfies the $\psi_2$ -condition, that is, there exist  two non negative real constants $B, b>0$ such that $ \mathbb{E}[e^{b \xi^2}] \leq B$.

\item A random variable  $\xi$  satisfies the Laplace transform condition, that is there exist two non negative real constants $B, b>0$ such that $\forall \lambda \in \mathbb{R}$,  $\ \  \mathbb{E}[e^{\lambda (\xi-\operatorname{E}[\xi])} ] \leq Be^{\lambda^2 b / 2}$. This condition is also referred to as the moment generating-condition, that is there exist two non negative real constants $B, b>0$ such that $ \mathbb{E}[ e^{t \xi} ] \leq B e^{t^2 b^2 /2 }$. The parameter $b$ is directly related to the variance proxy $\sigma$.

\item A random variable  $\xi$  satisfies  the Moment condition, that is there exists a non negative real constant $K>0$ such that $\  \forall p \geq 1 \ \left ( \mathbb{E}[ |\xi|^p \right ])^{1/p} \leq K \sqrt{p}$. It is easy to see with for instance Gaussian variables that $K$ can be expressed with respect to the variance proxy $\sigma$ as follows: $K= \sigma e^{1 / e}$ for $k \geq 2$ and $K = \sigma \sqrt{ 2 \pi}$.

\item A random variable  $\xi$  satisfies  the Union bound condition, that is there exists a non negative real constant $c>0$ such that $ \forall n \ge c \ \mathbb{E}[\max\{|\xi_1 - \operatorname{E}[\xi]|,\ldots,|\xi_n - \operatorname{E}[\xi]|\}] \leq c \sqrt{\log n}$  where $\xi_1, \ldots, \xi_n$ are independent and identically distributed random variables, copies of $\xi$.

\item The tail is less than the one of a Gaussian of variance proxy $\sigma$, there exist $b > 0$ and $Z \sim \mathcal{N}(0, \sigma^2)$ such that 
$\mathbb{P}(| \xi |>t) \leq b \mathbb{P}(| Z | \geq t )$. The latter definition explains the term sub-Gaussian constants quite well.
\end{itemize}
Obviously, the different negative real constants $B, b>0$ are not necessarily the same. 
\end{remark}

\begin{definition}
Referring to \cite{Tao_2013}, we say that an event $E$ holds with overwhelming probability if, for every fixed real constant
$k > 0$, we have 
\begin{equation}\label{overwhelming}
\mathbb{P}(E) \geq   1 - C_k / n^k
\end{equation}
\noindent for some constant $C_k$ independent of $n$ or equivalently 
$\mathbb{P}(E^{c}) \leq    C_k  e^{ -k \ln n } $ where $A^{c}$ denotes the complementary of $A$.
\end{definition}

\begin{remark}
Of course, the concept of overwhelming probability can be extended to a family of events $E_{\alpha}$ depending on some parameter $\alpha$ with the condition that each event in the family holds with overwhelming probability uniformly in $\alpha$ if the constant $C_k$ in the definition of overwhelming probability is independent of $\alpha$.
\end{remark}

\begin{remark}
Using Boole's inequality (also referred to as the union bound in the English mathematical literature) that states that the probability measure is $\sigma$-sub additive, we trivially see that if a family of events $E_{\alpha}$ of polynomial cardinality holds with overwhelming probability, then the intersection over $\alpha$ of this family $\bigcap \limits_{\alpha} E_{\alpha}$ still holds with overwhelming probability.
\end{remark}

\begin{remark}
The previous Boole's inequality remark emphasizes that although the concept of overwhelming probability is not the same as the one of almost surely, it is still something with very high probability. In the rest of the paper, we will even get tighter bound and prove
\begin{equation}\label{overwhelming2}
\mathbb{P}(E^{c}) \leq    C_k  e^{ -k n } 
\end{equation}
which implies that the event $E$ holds with overwhelming probability.
\end{remark}

\section{Upper bound for operator norm for sub-Gaussian tailed matrices}\label{main}
Equipped with these definition, we shall prove the following statement

\begin{proposition}\label{prop1}
Let a squared matrix $M$ be with independent coefficients $\xi_{i,j}$ with zero mean that are uniformly sub-Gaussian , then there exist non negative real constants $C,c > 0$ such that
\begin{equation}
\mathbb{P} ( \| M \|_{op} > A \sqrt n )\leq C \exp( -c A n) \label{eq1}
\end{equation}
\noindent for all $ A \geq C$. In particular, we have $\| M \|_{op} = O(\sqrt n )$
with overwhelming probability
\end{proposition}

\begin{proof}
See \ref{proof1}.
\end{proof}

\begin{remark}
This result is quite natural as the matrix coefficients $\xi_{i,j}$ are uniformly sub-Gaussian. Indeed in the proof, we have used the fact that the matrix coefficients $\xi_{i,j}$ $L_{\infty}$ norm was sub-Gaussian, hence any of the matrix row for the $L_a$ was sub-Gaussian. But can we go further and find a less stringent sufficient condition for the inequality \ref{eq1} to hold? The answer is yes and is provided by the condition stated in proposition \ref{prop2}.
\end{remark}

\begin{proposition}\label{prop2}
Let a squared matrix $M$ such that any of its row is uniformly sub-Gaussian for the norm $L_a$ and independent, then there exist non negative real constants $C,c > 0$ such that
\begin{equation}
\mathbb{P} ( \| M \|_{a,b} > A \sqrt n )\leq C \exp( -c A n) \label{eq2}
\end{equation}
\noindent for all $ A \geq C$. In particular, we have $\| M \|_{a,b} = O(\sqrt n )$
with overwhelming probability
\end{proposition}

\begin{proof}
See \ref{proof2}.
\end{proof}

\begin{remark}
If a random matrix has its rows uniformly sub-Gaussian, necessarily, any of its coefficients is also uniformly sub-Gaussian. This is trivially seen as for a given coefficient $\xi_{ij}$, the corresponding row $R_i$ is sub-Gaussian, hence there are positive constants $B, b$ that does not depend on $i$ such that for every $t > 0$,
\begin{equation}
\mathbb {P} (\| R_i \ |>t) \leq B \exp( {-b t^{2}} )
\end{equation}
Hence since   $\| R_i \ | > | xi_{ij} |$, we have as well 
\begin{equation}
\mathbb{P} ( | xi_{ij} | > t )\leq B \exp({ -b t^{2} } )
\end{equation}
which proves the uniform sub-Gaussian character of any of the matrix row.
The independence of the matrix rows, however, does not imply that each of the matrix row are independent, making the condition of proposition \ref{prop1} less stringent.
\end{remark}

\section{Conclusion}
This paper investigated an upper bound of the operator norm for sub-Gaussian tailed random matrices. We proved here that random matrices with independent rows that are uniformly sub-Gaussian satisfy the Tracy Widom bound, that is, the matrix operator norm remains bounded by $O(\sqrt n )$. An interesting extension would be to see how we can generalize our result to the $(\ell_p,\ell_r)$-Grothendieck problem, which seeks to maximize the bilinear form $y^T A x$ for an input matrix $A \in {\mathbb R}^{m \times n}$ over vectors $x,y$ with $\|x\|_p=\|y\|_r=1$. We know this problem is equivalent to computing the $p \to r^\ast$ operator norm of $A$, where $\ell_{r^*}$ is the dual norm to $\ell_r$. 

\bibliographystyle{plainnat}
\bibliography{mybib}
\ifnum\supp=1

\newpage
\appendix
\section{Proofs}
\subsection{Proof of proposition \ref{prop1}}\label{proof1}
We will do the proof thanks to three simple lemmas below that take advantage of the uniform sub-Gaussian tails bounds and the remarkable property of the Lipschitz character of the map $x \rightarrow \| M x \|$, combined with the compacity of the unit sphere.

Let us define the unit sphere $\mathcal{S} := \{ u \in \mathbb{R}^n | \| u \| = 1 \}$ of the $\mathbb{R}^n$ vector space. The result is similar for complex coefficients matrices in which case the unit sphere is modified into $\mathcal{S} := \{ u \in \mathbb{R}^n | \| u \| = 1 \}$ of the $\mathbb{C}^n$.
We will first prove the following lemma

\begin{lemma}\label{lemma1}
If the coefficients $\xi_{i,j}$ of $M$ are independent and have uniformly sub-Gaussian tails, then there exist absolute constants $C, c > 0$ such that for any $u \in \mathcal{S}$, we have
\begin{equation}\label{lemma1_eq}
\mathbb{P} ( \| M u \| > A \sqrt n )\leq C \exp( -c A n)
\end{equation}
for all $ A \geq C$.
\end{lemma}

\begin{proof}
Let $R_1, \ldots, R_n$ be the $n$ rows of the matrix $\mathbf{M}$, then the column vector $\mathbf{M} u$ has coefficients 
$R_i u $ for $i = 1, \ldots, n$. 

The matrix coefficients $\xi_{i,j}$ are all uniformly sub Gaussian, hence there are positive constants $B, b> 0$ independent of $i,j$ such that for every $t > 0$,
\begin{equation}
\mathbb{P}(  | \xi_{i,j} | \geq t ) \leq B \exp(-b t^2).
\end{equation}

This implies in particular that $R_i$ is also with sub-Gaussian tails but with different coefficients. This is because we have 
\begin{equation}
\mathbb{P}(  | R_i | \geq t ) \leq \mathbb{P}( \sqrt n \max_{j}| \xi_{i,j} | \geq t ) \leq B \exp(-\frac{b}{n} t^2).
\end{equation}
Hence taking $b'= \frac{b}{n}$, we have
\begin{equation}
\mathbb{P}(  \| R_i \|  \geq t ) \leq B \exp(-b' t^2).
\end{equation}

The Cauchy Schwartz inequality gives us that for $u \in \mathcal{S}$, we have $| R_i u |  \leq \| R_i  \|  \| u \|  = \| R_i  \|  $ as $\| u \| =1$, hence,
\begin{equation}
\mathbb{P}(  | R_i  u | \geq t ) \leq \mathbb{P}(  \| R_i  \| \geq t ) \leq B \exp(-b' t^2).
\end{equation}
which states that $R_i  u$ is uniformly sub-Gaussian or equivalently, that it satisfies the $\psi_2$ condition, that there exist two non negative constants $b, B >0$ (that are different constants from previously) such that:
\begin{equation}
\mathbb{E}[ e^{b | R_i  u | ^2} ] \leq B.
\end{equation}
Because of the assumption that the matrix coefficients are independent, each row $R_i  u$ is also independent and the vector $M u$ satisfies also the  $\psi_2$ condition as:
\begin{equation}
\mathbb{E}[ e^{b \| M u \| ^2} ] = \mathbb{E}[ \prod_{i=1}^n e^{b | R_i  u | ^2} ] =  \prod_{i=1}^n \mathbb{E}[  e^{b | R_i  u | ^2} ]  \leq B^n.
\end{equation}
Let us take $C= B^n$ and take $A \geq C$ and $n \geq 1$. The Markov property gives us
\begin{eqnarray}
\mathbb{P}( \| M u \| \geq A \sqrt{n}  )  = \mathbb{P}( e^{ b \| M u \|^2 } \geq  e^{ b \, A^2  n } )  \leq \frac{ \mathbb{E}[ e^{b \| M u \| ^2} ] } { e^{ b \, A^2 n } }  \leq C e^{ -b \, A^2 n } \leq C e^{ -b \, C A n } 
\end{eqnarray}
Taking $c = b \,C$, we get the required inequality: 
$$
\mathbb{P} ( \| M u \| > A \sqrt n )\leq C \exp( -c A n)
$$ which concludes the proof.
\end{proof}

\begin{remark}
Expressing the lemma \ref{lemma1} in terms of probability, we have proved that for any individual unit vector $u$, the norm of the matrix multiplication of $M$ with $u$, denoted by $\| M u \|$ is with growth at most $\sqrt n$ or equivalently  $\| M u \| = O(\sqrt n )$ with overwhelming probability.
\end{remark}

\begin{remark}
At this stage, we could imagine that equipped with lemma \ref{lemma1}, we could finalize the proof of proposition \ref{prop1}. The slight difference between lemma \ref{lemma1} and proposition \ref{prop1} is the applying set. Lemma \ref{lemma1} states that for any individual unit vector $u$, the norm of the matrix multiplication of $M$ with $u$, denoted by $\| M u \|$ is with growth at most $\sqrt n$. Proposition \ref{prop1} states that the supremum over the unit sphere of any individual unit vector $u$, the norm of the matrix multiplication of $M$ with $u$, denoted by $\| M u \|$ is with growth at most $\sqrt n$. 
We could imagine going from lemma  \ref{lemma1} to proposition \ref{prop1} using the simple union bound on all points of the unit sphere for the operator norm as follows:
\begin{equation}
\mathbb{P}(  \| \mathbf{M} \|_{op} > \lambda ) \leq \mathbb{P}( \bigcup \limits_{u \in \mathcal{S} } \|  \mathbf{M} u \| > \lambda )
\end{equation}
However, we would be stuck as the  unit sphere $\mathcal{S}$ is an uncountable number of points set.

To solve this issue, we shall change the set in the union bound and use the usual trick of maximal $\varepsilon$-net of the unit sphere $\mathcal{S}$, denoted by $\Sigma(\varepsilon)$. This leads to lemma \ref{lemma2}. As we will see shortly, the maximal $\varepsilon$-net of the sphere $\mathcal{S}$ is countable, using standard packing arguments. On this particular set, we can exploit the fact that the map $x \rightarrow \| M x \|$ is Lipschitz with Lipschitz constant given by $\| M \|_{op}$. The induced continuity also us controlling the upper bound of the norm of $\| M v \|$ for $v \in \Sigma(\varepsilon)$.
\end{remark}

\begin{lemma}\label{lemma2}
Let $0 < \varepsilon < 1$ and $\Sigma(\varepsilon)$ be the maximal $\varepsilon$-net of the sphere $\mathcal{S}$, that is the set of points in $\mathcal{S}$ separated from each other by a distance of at least $\varepsilon$ and which is maximal with respect to set inclusion. Then for any $n \times n$ matrix $M$ and any $\lambda > 0$, we
have
\begin{equation}\label{lemma2_eq}
\mathbb{P}(  \| M  \|_{op} > \lambda) \leq \mathbb{P}( \bigcup \limits_{v \in \Sigma(\varepsilon) } \|  \mathbf{M} v \| > \lambda (1-\varepsilon) )
\end{equation}
\end{lemma}

\begin{proof}
From the definition of the operator norm (see \ref{operatornorm_def}) as a supremum,  using the fact that the map $x \rightarrow \| M x \|$ is Lipschitz, hence continuous and that the unit sphere  $\mathcal{S}$ is compact as we are in finite dimension, we can find  $x \in \mathcal{S}$ such that it attains the supremum (recall that a continuous function attains its supremum on a compact set).

\begin{equation}
\| M x \| = \| M \|_{op}
\end{equation}
We can eliminate the trivial case of $x$ belonging to $\Sigma(\varepsilon)$ 
as the inequality \ref{lemma2_eq} is easily verified in this scenario. 
In the other case, 
where $x$ does not belong to $\Sigma(\varepsilon)$, there must exist a point $y$ in $\Sigma(\varepsilon)$ 
whose distance to $x$ is less than $\varepsilon$ (otherwise we would have a contradiction of the maximality of $\Sigma(\varepsilon)$  
by including $x$ to $\Sigma(\varepsilon)$). 
We are going now to use the Lipschitz feature of the map  $x \rightarrow \| M x \|$ whose Lipschitz constant given by $\| M \|_{op}$
Since $\| x- y \| \leq \varepsilon$, 
the Lipschitz property gives us
\begin{equation}
\| M (x-y) \| \leq \| M \|_{op}\| x- y \|\leq  M \|_{op}\|  \varepsilon
\end{equation}
The triangular inequality gives us
\begin{equation}
 \| M x \|_{op}  = \| M x \| \leq  \| M (x-y) \|  + \| M y \|  \leq  \ \|  M \|_{op} \varepsilon + \| M y \|
\end{equation}
Hence,
\begin{equation}
 \| M y \| \geq M \|_{op} (1-\varepsilon)
\end{equation}
In particular if  $\| M y \|_{op} > \lambda$, then $\| M y \| >  \lambda (1-\varepsilon)$ which concludes the proof
\end{proof}

\begin{remark}
The lemma \ref{lemma2_eq} is very intuitive. The continuity of the  map  $x \rightarrow \| M x \|$ implies that it attains its maximum on the compact unit sphere.  By packing argument, we have necessarily that around this optimum, there is a point of the maximum set $\Sigma(\varepsilon)$ with a distance lower than $\varepsilon$. As  the map  $x \rightarrow \| M x \|$ is Lipschitz, the decrease between the optimum and this point in $x \rightarrow \| M x \|$ should be at most $ \|  M \|_{op} \varepsilon $ as $ \|  M \|_{op} $ is the constant Lipschitz.
\end{remark}

We recall last but not least that the cardinality of the maximal $\varepsilon$-net of the sphere $\mathcal{S}$, $\Sigma(\varepsilon)$ should be polynomial at most in $n-1$, the dimension of the sphere with the following two lemmas

\begin{lemma}\label{lemma3}
Let $0 < \varepsilon < 1$, and let $\Sigma(\varepsilon)$ be a maximal $\varepsilon$-net of the unit sphere $\mathcal{S}$. Then $\Sigma(\varepsilon)$  has cardinality at most
$\frac{C }{\varepsilon^{n-1}} $ for some non negative constant $C > 0$ and at least $\frac{c}{\varepsilon^{n-1}}$  for some constant $c >0$.
\end{lemma}

\begin{proof}
The proof is quite intuitive and simple. It relies on a volume packing argument. The balls of radius $\varepsilon /  2$ centered around each point of  $\Sigma(\varepsilon)$ are disjoint and they are in the same numbers as the cardinal of $\Sigma(\varepsilon)$. By the triangular inequality, and using the fact that $\varepsilon  / 2 <1$, all these balls are contained within the intersection of the large ball of radius $(1+\varepsilon  / 2)$ and center the origin, and the smaller ball of radius $(1-\varepsilon / 2)$ and center the same origin. Hence, (using the fact that the volume of a ball is a constant times the radius to the power the dimension of the space), we can pack at most 
$$
\frac{ (1+\varepsilon  / 2 )^{n} - (1-\varepsilon  /2)^{n} }{ (\varepsilon  /2)^n}
$$
of these balls, which proves that the cardinality is at most $(C / \varepsilon)) ^{n-1} $ for some non negative constant $C > 0$ as the constant is for $\varepsilon $ small equivalent to $\frac{2 n} {\varepsilon ^{n-1}}$.

Reciprocally, for $\varepsilon < 2$, the $\Sigma(\varepsilon)$ is not empty. If we sequentially pack the space between the same previous large and small balll of radius $(1+\varepsilon  / 2)$ and $(1-\varepsilon / 2)$ respectively, both centered at the origin, with balls that do not intersect and have radius $\varepsilon  / 2$ and with centers on the unit sphere, we can take the set of the centers of these balls. As the balls of radius $\varepsilon  / 2$ do not intersect, by the triangular inequality, their centers are at least at a distance greater or equal to $\varepsilon$. Because $\Sigma(\varepsilon)$ is a maximal set, its cardinality should be at least equal to the number of previously created centers. We can pack 
$$
\frac{ (1+\varepsilon  / 2 )^{n} - (1-\varepsilon  /2)^{n} }{ (\varepsilon  /2)^n}
$$
of these centers, which proves that the cardinality is at least  $(c / \varepsilon)) ^{n-1} $ for some non negative constant $c > 0$ 
\end{proof}

\begin{proof}
We can now prove proposition \ref{prop1} as follows. 
Using lemma \ref{lemma2} and the union bound, we have
\begin{equation}
\mathbb{P}(  \| M  \|_{op} > A \sqrt{n}) \leq \mathbb{P}( \bigcup \limits_{v \in \Sigma(\varepsilon) } \|  \mathbf{M} v \| > A \sqrt{n} (1-\varepsilon) ) \leq \sum_{v \in \Sigma(\varepsilon)}  \mathbb{P}(  \|  \mathbf{M} v \| > A \sqrt{n} (1-\varepsilon)  )
\end{equation}

Lemma \ref{lemma1} states that for $v \in \mathcal{S}$, there exist absolute constants $C,c > 0$ such that 
\begin{equation}
\mathbb{P}(  \|  \mathbf{M} v \| > A \sqrt{n} (1-\varepsilon)  ) \leq C \exp( -c A  (1-\varepsilon) n )
\end{equation}

Since the cardinality of $\Sigma(\varepsilon) $ is bounded by $\frac{K }{\varepsilon^{n-1}}$, we can upper bound $\mathbb{P}(  \| M  \|_{op} > A \sqrt{n})$ by
\begin{equation}
\mathbb{P}(  \| M  \|_{op} > A \sqrt{n}) \leq \frac{K }{\varepsilon^{n-1}} C \exp( -c A  (1-\varepsilon) n ) 
\end{equation}

Fixing $\varepsilon = 1 / 2$, denoting by $C ' = K C $ and taking $c'$ such that $c'A = cA / 2 - \ln 2$, we have
\begin{equation}
\mathbb{P}(  \| M  \|_{op} > A \sqrt{n}) \leq C' exp( -c' A n )
\end{equation}
which concludes the proof.
\end{proof}

\subsection{Proof of proposition \ref{prop2}}\label{proof2}
\begin{proof}
This is exactly the same reasoning as proposition \ref{prop1} but starting with the fact that any of the matrix $M$ row is uniformly sub-Gaussian  for the norm $L_a$. This means that there exist absolute constants $B,b>0$ such that for any $i=1, \ldots,n$
\begin{equation}
\mathbb{P}( \| R_i \|_a  \geq t ) \leq B \exp(-b t^2).
\end{equation}
for the norm $L_a$. The independence of the rows allows us proving the following lemma (similar to lemma \ref{lemma1}) that under the condition of proposition \ref{prop2}, there exist absolute constants $C, c > 0$ such that for any $u \in \mathcal{S}$, we have
\begin{equation}\label{lemma1_eq2}
\mathbb{P} ( \| M u \| > A \sqrt n )\leq C \exp( -c A n)
\end{equation}
for all $ A \geq C$. lemma \ref{lemma2} and \ref{lemma3} remain unchanged allowing to conclude.
\end{proof}

\end{document}